\documentclass[oneside,english]{amsart}
\usepackage[T1]{fontenc}
\usepackage[latin9]{inputenc}
\usepackage{babel}
\usepackage{enumitem}
\usepackage{amstext}
\usepackage{amsthm}
\usepackage{amssymb}
\usepackage[unicode=true]
 {hyperref}

\makeatletter
\numberwithin{equation}{section}
\numberwithin{figure}{section}
\theoremstyle{plain}
\newtheorem{thm}{\protect\theoremname}
\theoremstyle{definition}
\newtheorem{defn}[thm]{\protect\definitionname}
\theoremstyle{plain}
\newtheorem{prop}[thm]{\protect\propositionname}
\theoremstyle{definition}
\newtheorem{example}[thm]{\protect\examplename}
\theoremstyle{remark}
\newtheorem{rem}[thm]{\protect\remarkname}
\theoremstyle{plain}
\newtheorem{lem}[thm]{\protect\lemmaname}
\newlist{casenv}{enumerate}{4}
\setlist[casenv]{leftmargin=*,align=left,widest={iiii}}
\setlist[casenv,1]{label={{\itshape\ \casename} \arabic*.},ref=\arabic*}
\setlist[casenv,2]{label={{\itshape\ \casename} \roman*.},ref=\roman*}
\setlist[casenv,3]{label={{\itshape\ \casename\ \alph*.}},ref=\alph*}
\setlist[casenv,4]{label={{\itshape\ \casename} \arabic*.},ref=\arabic*}

\makeatother

\providecommand{\casename}{Case}
\providecommand{\definitionname}{Definition}
\providecommand{\examplename}{Example}
\providecommand{\lemmaname}{Lemma}
\providecommand{\propositionname}{Proposition}
\providecommand{\remarkname}{Remark}
\providecommand{\theoremname}{Theorem}

\begin{document}
\title[Generating the Group of Nonzero Elements]{Generating the Group of Nonzero Elements of a Quadratic Extension
of $\mathbb{F}_{p}$}
\author{Jerry D Rosen}
\address{Professor, Department of Mathematics, California State University,
Northridge, 18111 Nordhoff St., Northridge, Ca 91311}
\email{jerry.rosen@csun.edu}
\author{Daniel Sarian}
\address{Graduate Student, Department of Mathematics, California State University,
Northridge, 18111 Nordhoff St., Northridge, Ca 91311}
\email{daniel.sarian.532@my.csun.edu}
\author{Susan Elizabeth Slome}
\address{Independent Researcher, 237 Old Willets Path, Smithtown, NY 11787}
\email{sslome1234@gmail.com (Corresponding Author)}
\date{February 15, 2022}
\begin{abstract}
It is well known that if $\mathbb{F}$ is a finite field then $\mathbb{F^{*}}$,
the set of nonzero elements of $\mathbb{F}$, is a cyclic group. In
this paper we will assume $\mathbb{F}=\mathbb{F}_{p}$ (the finite
field with $p$ elements, $p$ a prime) and $\mathbb{\mathbb{F}}_{p^{2}}$
is a quadratic extension of $\mathbb{F}_{p}$. In this case, the groups
$\mathbb{F}_{p}^{*}$ and {\normalsize{}$\mathbb{F}_{p^{2}}^{*}$}
have orders $p-1$ and $p^{2}-1$ respectively. We will provide necessary
and sufficient conditions for an element $u\in\mathbb{F}_{p^{2}}^{*}$
to be a generator. Specifically, we will prove $u$ is a generator
of $\mathbb{F}_{p^{2}}^{*}$ if and only if $N(u)$ generates $\mathbb{F}_{p}^{*}$
and $\text{\ensuremath{\frac{u^{2}}{N(u)}}}$ generates $Ker\,N$,
where $N:\mathbb{F}_{p^{2}}^{*}\rightarrow\mathbb{F}_{p}^{*}$ denotes
the norm map. 
\end{abstract}

\thanks{We would like to thank the referee for their valuable suggestions
which greatly improved this paper.}
\maketitle

\section{\label{sec:background}background}

Our paper has several purposes: a major topic in basic number theory
is primitive roots. If we let $U(n)$ denote all those classes $[k]$,
where $k$ and $n$ are relatively prime, then it is well known this
is a multiplicative group. If this group is cyclic, then the generators
are called primitive elements $mod$ $n.$ The classification of those
natural numbers having primitive roots occupies a significant portion
of elementary number theory. The classification theorem states: a
natural number $n$ will have a primitive root if and only if $n=2,4,p^{k},2p^{k},$
for some odd prime $p$ and some $k\geq1$. As a special case, the
nonzero elements of a finite field have primitive roots. We note that
there is no result which gives a formula stating which integers are
primitive roots for a given $n$. If $L/K$ is a finite extension
of finite fields, one can ask: can generators for $L^{*}$ be determined
by generators of $K^{*}$ and other subgroups? We were able to settle
this question for quadratic extensions of $\mathbb{F}{}_{p}$ and,
interestingly enough, the norm map plays a significant role. However,
beyond the result, we believe there are some didactic opportunities
embedded in our approach. We will use some basic ideas from abstract
algebra and number theory. Several of our results demonstrate how
group homomorphisms are a useful tool for transferring information
between groups. In addition, we use the notion of a quadratic residue
in several proofs. Hence, students reading this article can see some
basic ideas in action. Furthermore, experimentation is a fundamental
tool in all areas of mathematics research. The statement of our main
result was arrived at after much experimenting. We used computer code
to help us to formulate and verify our main result. Part of that code
can be accessed in a Python file located at \href{https://github.com/sslome/generator_util}{https://github.com/sslome/generator\_util}. 

Suppose K is a finite field. It is well known that (see reference
\cite{key-1})
\begin{enumerate}
\item $K$ is a vector space over $\mathcal{\mathbb{F}}_{p}$ (the field
of integers $mod\,p$, for a prime $p$)
\item $K$ has $p^{n}$ elements for some natural number $n$
\item $K^{*}$ the nonzero elements of $K$ is a cyclic group of order $p^{n}-1$ 
\end{enumerate}
If $K=\mathbb{\mathbb{F}}_{p^{2}}$ is a quadratic extension of $\mathbb{F}_{p}$,
then $K=\mathbb{F}_{p}(w)$, where $w$ is algebraic over $\mathbb{F}_{p}$
of degree $2$. We will make use of the language of number theory
throughout. 
\begin{defn}
Let $p$ be a prime integer. A nonzero integer $n$ is called a quadratic
residue (mod$\,p$) if $n\equiv a^{2}mod\,p$ for some integer $a$.
So the quadratic residues are simply the elements in $\mathbb{F}_{p}$
which are squares. An element which is not a quadratic residue is
called a quadratic nonresidue. Let $Q(p)$ denote the set of quadratic
residues $mod\,p$. 
\end{defn}

The next result counts the number of quadratic residues. There is
an easy number theoretic proof of this (see reference \cite{key-1}),
however we will show how homomorphisms can provide a method for counting.
\begin{prop}
For any prime $p>2$, there are $\frac{p-1}{2}$ quadratic residues
and $\frac{p-1}{2}$ quadratic non-residues.
\end{prop}

\begin{proof}
Define $f:\mathbb{F}_{p}^{*}\rightarrow Q(p)$ by $f(x)=x^{2}$. Then
$f$ is a group epimorphism. Now $x\in Ker\,f$ if and only if $x^{2}\equiv1\,mod\,p$
if and only if $p\mid(x^{2}-1)$ if and only if $x\equiv\pm1\,mod\,p$.
Thus $Kerf=\left\{ \pm1\right\} $ and by the fundamental homomorpism
theorem $\mathbb{F}_{p}^{*}/\left\{ \pm1\right\} \cong Q(p)$, and
thus $Q(p)$ must contain $\frac{p-1}{2}$ elements and since $\mathbb{F}_{p}^{*}$
has $p-1$ elements, the number of quadratic nonresidues is $(p-1)$$-\frac{p-1}{2}$$=$$\frac{p-1}{2}$. 
\end{proof}
\begin{defn}
Let $p>2$ be an odd prime and $a$ an integer such that $p$$\nmid a$
. The Legendre symbol is defined as follows: $\left(\frac{a}{p}\right)=\begin{cases}
1 & if\,a\,is\,a\,quadratic\,residue\,mod\,p\\
-1 & if\,a\,is\,a\,quadratic\,non\:residue\,mod\,p
\end{cases}$ 
\end{defn}

The Legendre symbol has many significant properties, the primary one
we will require is contained in the next result. (For a proof see
reference \cite{key-1})
\begin{prop}
\label{prop:legendre}Let $p>2$ be a prime and let $a,b$ be integers.
Then $\left(\frac{ab}{p}\right)=\left(\frac{a}{p}\right)\left(\frac{b}{p}\right)$. 
\end{prop}

As we stated, the set of nonzero elements of a finite field is a cyclic
group. A generator for such a group is called a primitive element.
Hence, the group $\mathbb{F}_{p}^{*}$ has primitive elements. We
remark that a quadratic residue cannot be a primitive element. (see
reference \cite{key-1})

Another concept we will require, which comes from the theory of finite
dimensional extension fields, is the norm (see reference \cite{key-2}).
We will briefly describe this concept in more generality than we require.
Suppose $L$$/K$ is a finite extension of fields and let $u\in L$.
Define $T_{u}:L\rightarrow L$ by $T_{u}(x)=ux$, which is clearly
a linear transformation. Fix any basis for $L$$/K$ and let $N_{L/K}:L\rightarrow K$
by $N_{L/K}(u)=det\,T_{u}$. $N_{L/K}$ is called the norm map and
it is clearly multiplicative. In the following example we work this
out for the case we will require.
\begin{example}
Let $K=\mathbb{F}_{p}$ and $L=\text{\ensuremath{\mathbb{\mathbb{F}}_{p^{2}}}\ensuremath{=\mathbb{F}{}_{p}} (}w)$
be a quadratic extension. So $w=\sqrt{n}$ where $n$ is a quadratic
nonresidue $mod$ $p$. Let $u=a+b\sqrt{n}$ and $T_{u}$ be defined
as above. Then $T_{u}$$(1)=a+b\sqrt{n}$ and $T_{u}(\sqrt{n})=nb+a\sqrt{n}$.
Hence the matrix of $T_{u}$$=$$\left[\begin{array}{cc}
a & nb\\
b & a
\end{array}\right]$ with $det$$\,T_{u}=$$N_{L/K}(u)=a^{2}-nb^{2}$. From now on we
will just write $N:\mathbb{F}_{p}^{*}\left(\sqrt{n}\right)\rightarrow\mathbb{F}_{p}^{*}$
where $N(a+b\sqrt{n})=a^{2}-nb^{2}$ and this is a homomorphism of
groups (of course, the fact that $N$ is multiplicative could have
been proven directly). In the next section we will prove $N$ is onto. 
\end{example}

\section{The main Result}

In this section we assume $p$ is an odd prime and $n$ is a quadratic
nonresidue $mod$$\,p$ ($n$ is an integer which either equals $-1$
or is greater than or equal to $2$). As discussed in section \ref{sec:background},
$\mathbb{F}_{p^{2}}^{*}$ is a cyclic group of order $p^{2}-1$. The
main result of this paper Theorem \ref{thm:main result}, classifies
the generators of $\mathbb{F}_{p^{2}}^{*}$ in terms of generators
of $Ker\,N$ and $\mathbb{F}_{p}^{*}.$ Throughout this section we
will let $\mathbb{F}_{p^{2}}^{*}$ denote the nonzero elements of
the quadratic extension $\mathbb{F}{}_{p}$$(\sqrt{n})$ ($n$ a quadratic
nonresidue $mod$ $p$).
\begin{prop}
\label{prop:norm map}Suppose $p$ is an odd prime and $n$ is a quadratic
nonresidue $mod\,p$. Then the map \textup{$N:\mathbb{F}_{p^{2}}^{*}\rightarrow\mathbb{F}_{p}^{*}$
}given by $N(a$+b$\sqrt{n})=a^{2}-nb^{2}$ is an onto group homomorphism. 
\end{prop}

\begin{proof}
The fact that $N$ is a homomorphism was discussed in section \ref{sec:background}.
We will prove it is onto.

For any $w\in\mathbb{F}_{p}^{*}$ define $h:\mathbb{F}_{p}^{*}\rightarrow\mathbb{F}_{p}^{*}$
by $h(x)=wx.$ Note that $h$ is bijective. Choose the smallest integer
$k>0$ such that $wk$ is a quadratic nonresidue. Then $wk=w(k-1)+w=c^{2}+w$,
($c^{2}=w(k-1)),$ for some $c.$ Now, choosing $w=-n$ we have that
$-nk=c^{2}-n$ and we have the bijective map $g:\mathbb{F}_{p}^{*}\rightarrow\mathbb{F}_{p}^{*}$
defined by $g(x)=x(c^{2}-n).$ We know, $c^{2}-n$ is a quadratic
nonresidue and by Proposition \ref{prop:legendre}, whenever $x$
is a quadratic residue $g(x)$ is a quadratic nonresidue. So, as $x$
varies among all squares, $g(x)$ maps onto all quadratic nonresidues
in $\mathbb{F}_{p}^{*}.$

Returning to $N$, $N(a)=a^{2}$ and so $N$ maps onto all quadratic
residues. On the other hand, $N(ac+a\sqrt{n})=N(a(c+\sqrt{n}))=a^{2}(c^{2}-n)$
and so $N$ also maps onto all quadratic nonresidues.
\end{proof}
\begin{rem}
By The Fundamental Homomorphism Theorem $|Ker\,N|=p+1.$
\end{rem}

\begin{lem}
\label{lem:kernel}For any $u\in\mathbb{F}_{p^{2}}^{*}$, $\frac{u^{2}}{N(u)}\in Ker\,N$.
\end{lem}

\begin{proof}
Note that for any $c\in\mathbb{Z},$ $N(cu)=N(ca+cb\sqrt{n})=(ca)^{2}-n(cb)^{2}=c^{2}N(u)$.
Hence
\[
N\left(\frac{u^{2}}{N(u)}\right)=\frac{1}{N(u)^{2}}(N(u))^{2}=1
\]
\end{proof}
\begin{lem}
\label{homomorphism}Suppose $G$ and $H$ are multiplicative groups,
with $H$ abelian. If $g,h:G\rightarrow H$ are homomorphisms, then
the map $f:G\rightarrow H$ defined by $f(x)=g(x)(h(x))^{-1}$ is
a homomorphism.
\end{lem}

\begin{proof}
For any $x,y\in G,$ $f(xy)=g(xy)(h(xy))^{-1}=g(x)g(y)(h(x)h(y))^{-1}=g(x)g(y)(h(y))^{-1}(h(x))^{-1}$
and since $H$ is abelian 
\end{proof}
$g(x)g(y)(h(y))^{-1}(h(x))^{-1}=g(x)(h(x))^{-1}g(y)(h(y))^{-1}=f(x)f(y).$
\begin{prop}
\label{prop:kernel map}Let $p$ be an odd prime and $n$ a quadratic
nonresidue. Then the map \textup{$f:\mathbb{F}_{p^{2}}^{*}\rightarrow Ker\,N$
given by $f(u)=\frac{u^{2}}{N(u)}$ }

is an onto group homomorphism.
\end{prop}

\begin{proof}
Note that the maps $g:\mathbb{F}_{p^{2}}^{*}\rightarrow\mathbb{F}_{p^{2}}^{*}$
given by $g(u)=u^{2}$ and $N$ are homomorphisms and so the fact
that $f$ is a homomorphism follows from Lemma \ref{homomorphism}.

If $a\in\mathbb{F}_{p}^{*}$ then $f(a)=\frac{a^{2}}{N(a)}=1.$ On
the other hand, if $u=a+b\sqrt{n}\in Ker\,f$, then $u^{2}=N(u)$
implies $a^{2}+nb^{2}=a^{2}-nb^{2}$ which gives $nb^{2}=0$, giving
$b=0$ and thus $u=a\in$$\mathbb{F}_{p}^{*}$ . We conclude that
$Ker\,f=\text{\ensuremath{\mathbb{F}_{p}^{*}} }$ and we obtain $\mathbb{F}_{p^{2}}^{*}/\mathbb{F}_{p}^{*}\cong im\,f.$
Therefore, $|im\,f|=p+1=|Ker\,N|$, proving $f$ is onto. 
\end{proof}
\begin{lem}
\label{lem:p+1,p-1}If $p$ is an odd prime, then $2$ is greatest
common divisor of $p-1$ and $p+1.$
\end{lem}

\begin{proof}
If $k$ is any divisor of $p-1$ and $p+1,$ then $k$ divides $(p+1)-(p-1)=2.$
\end{proof}
We now turn to the main result of the paper.
\begin{thm}
\label{thm:main result}Suppose $p$ is an odd prime and $n$ an integer
such that $n=-1$ or $n\geq2$ and assume $n$ a quadratic nonresidue
$mod\,p.$ Then $u=a+b\sqrt{n}$ is a generator of \textup{$\mathbb{F}_{p^{2}}^{*}$
if and only if $N(u)=a^{2}-nb^{2}$ is a generator of $\mathbb{F}_{p}^{*}$
(i.e. a primitive root $mod\,p)$ and $f(u)=\frac{u^{2}}{N(u)}$ is
a generator of the $Ker\,N.$ }
\end{thm}

\begin{proof}
If $u$ is a generator of $\mathbb{F}_{p^{2}}^{*}$ then the onto
homomorphisms $N$ and $f$ (see propositions \ref{prop:norm map}
and \ref{prop:kernel map}) imply $N(u)$ and $f(u)$ are generators
of $\mathbb{F}_{p}^{*}$ and $Ker\,N,$ respectively.

Conversely, assume $N(u)$ and $f(u)$ are generators for $\mathbb{F}_{p}^{*}$
and $Ker\,N$. Then the $ord(N(u))=p-1$ and $ord$$\left(f(u)\right)=p+1.$
We know $ord(u)|p^{2}-1,$ our goal is to prove they are equal. 

Let $u=a+b\sqrt{n}$ and set $\hat{u}=a-b\sqrt{n}$. Then the map
$\varphi(u)=\hat{u}$ is an automorphism of $\mathbb{F}{}_{p}$. Notice,
$N(u)=u\hat{u}$ and $f(u)=\frac{u^{2}}{N(u)}=\frac{u}{\hat{u}}$.
If $u^{k}=1$ then $\hat{u}^{k}=1$ and so $(N(u))^{k}=(u\hat{u})^{k}=1.$
Also, $(f(u))^{k}=(\frac{u}{\hat{u}})^{k}=1.$ Hence, the orders of
$N(u)$ and $f(u),$ respectively $p-1$ and $p+1$, must divide $k$
and so the $ord(u)$ is a multiple of both $p-1$ and $p+1.$ By Lemma
\ref{lem:p+1,p-1}, the $ord(u)$ is either $p^{2}-1$ or $\frac{1}{2}(p^{2}-1).$
We must prove that it is $p^{2}-1.$

For $v=r+s\sqrt{n},$ if $1=v^{2}=(r^{2}+ns^{2})+2rs\sqrt{n}$ then
$r=0$ or $s=0.$ If $s=0,$ then $v=r=\pm1.$ If $r=0,$ then $ns^{2}=1,$
which contradicts the assumption that $n$ is a quadratic nonresidue. 

If $ord(u)=\frac{1}{2}(p-1)(p+1)$ then $1=(u^{\frac{1}{4}(p-1)(p+1)})^{2}=(\hat{u}^{\frac{1}{4}(p-1)(p+1)})^{2}.$
Thus, $u^{\frac{1}{4}(p-1)(p+1)}=\hat{u}^{\frac{1}{4}(p-1)(p+1)}=\pm1.$
Hence, 
\[
(N(u))^{\frac{1}{4}(p-1)(p+1)}=(f(u))^{\frac{1}{4}(p-1)(p+1)}=1
\]
.
\begin{casenv}
\item Assume $p\equiv1mod\:4.$ In this case $p-1$ is divisible by $4$
and we can write $1=(N(u))^{\frac{1}{4}(p-1)(p+1)}=((u\hat{u})^{p+1})^{\frac{1}{4}(p-1)}=(u\hat{u})^{\frac{1}{2}(p-1)},$
since $(u\hat{u})^{p+1}=(u\hat{u})^{2},$ but this contradicts the
fact that the $ord(N(u))$ is $p-1.$ 
\item Assume $p\equiv3mod\:4.$ In this case $p+1$ is divisible by $4$
and we can write $1=((\frac{\hat{u}}{u})^{p-1})^{\frac{1}{4}(p+1)}=(\frac{u}{\hat{u}})^{\frac{1}{2}(p+1)},$
since $(\frac{\hat{u}}{u})^{(p-1)}=(\frac{\hat{u}}{u})^{(p+1)}(\frac{u}{\hat{u}})^{2}$
and this contradicts the fact that $ord(f(u))=p+1.$ 
\end{casenv}
By cases $1$ and $2$ the order of $u$ cannot be $\frac{1}{2}(p^{2}-1)$
and so must be $p^{2}-1$ and thus $u$ generates $\mathbb{F}_{p^{2}}^{*}.$
\end{proof}

\end{document}